\documentclass[a4paper,12pt]{article}
\pdfoutput=1
\usepackage{amsthm, amsmath, amssymb, xcolor, lmodern, graphicx, import, caption, subcaption,tikz,amssymb,authblk,amsbsy,bm,comment}
\usepackage[letterpaper,top=2cm,bottom=2cm,left=3cm,right=3cm,marginparwidth=1.75cm]{geometry}
\usepackage[colorlinks=true, allcolors=blue]{hyperref}

\theoremstyle{plain}
\newtheorem{theorem}{Theorem}

\newtheorem{lemma}[theorem]{Lemma}
\newtheorem{proposition}[theorem]{Proposition}
\newtheorem{case}{Case}
\theoremstyle{definition}
\newtheorem{definition}[theorem]{Definition}
\theoremstyle{remark}

\newcommand{\vocab}[1]{\textbf{\color{blue!50!gray}#1}}
\newcommand{\p}{\mathcal{OP}}

\newcommand{\df}{\draw[fill]}

\newcommand{\bt}{\begin{tikzpicture}}
\newcommand{\et}{\end{tikzpicture}}
\providecommand{\keywords}[1]
{
  \small	
  \textbf{\textit{Keywords---}} #1
}

\title{Outerplanar Tur\'an number of a cycle}
\author{Ervin Gy\H{o}ri \thanks{R\'enyi Institute, Budapest,
 Hungary.  Research partially supported by the NKFIH Grant 132696. E-mail: {\tt gyori.ervin@renyi.hu}}
\and  Guilherme~Zeus Dantas~e~Moura \thanks{Haverford College, PA 19041, USA. E-mail: gdantasemo@haverford.edu; zeusdanmou@gmail.com} 
\and Runtian Zhou \thanks{
Davidson College, Davidson, North Carolina 28035
E-mail: {\tt dazhou@davidson.edu}}
}

\begin{document}
\maketitle
\begin{abstract}
A graph is outerplanar if it has a planar drawing for which all vertices belong to the outer face of the drawing. Let $H$ be a graph. The outerplanar Tur\'an number of $H$, denoted by $ex_\p(n,H)$, is the maximum number of edges in an $n$-vertex outerplanar graph which does not contain $H$ as a subgraph. In 2021, L. Fang et al. determined the outerplanar Tur\'an number of cycles and paths. In this paper, we use techniques of dual graph to give a shorter proof for the sharp upperbound of $ex_\p(n,C_k)\leq \frac{(2k - 5)(kn - k - 1)}{k^2 - 2k - 1}$.
\end{abstract}

\section{Introduction and Main Results}
In this paper, all graphs considered are outerplanar, undirected, finite and contain neither loops nor multiple edges. More specifically, we study outerplanar plane graphs what are embeddings (drawings) of graphs in the plane such that the edge curves do not cross each other; they may share just endpoints, and that all vertices belong to the outerface. We use $C_k$ to denote the cycle of $k$ vertices. We use $n$-face to denote a face with $n$ edges. In particular, we use $(4+)$-innerface to denote an innerface of some outerplanar plane graph with at least $4$ edges. We denote the vertex and the edge sets of a graph $G$ by $V(G)$ and $E(G)$ respectively. We also denote the number of vertices and edges of $G$ by $v(G)$ and $e(G)$ respectively. The minimum degree of $G$ is denoted $\delta(G)$.

The Tur\'an number $ex(n,H)$ for a graph $H$ is the maximum number of edges in an $n$-vertex graph with no copy of $H$ as a subgraph. The first result on the topic of Tur\'an number was obtained by Mantel and Tur\'an, who proved that the balanced complete $r$-partite graph is the unique extremal graph of $ex(n,K_{r+1})$ edges. The Erd\H{o}s-Stone-Simonovits theorem \cite{erdos1963structure,erdos1962number} then generalized this result and asymptotically determined $ex(n,H)$ for all nonbipartite graphs $H:$ $ex(n,H)=(1-{1\over \mathcal{X}(H)-1}){n\choose 2}+o(n^2)$.

In 2016, Dowden et al. \cite{dowden2016extremal} initiated the study of Tur\'an-type problems when host graphs are plane graphs, i.e., how many edges can a plane graph on $n$ vertices have, without containing a given graph as a subgraph? Let $\mathcal{H}$ be a set of graphs. The planar Tur\'an number, $ex_\p(n,\mathcal{H})$, is the maximum number of edges in an $n$-vertex planar graph which does not contain any member of $\mathcal{H}$ as a subgraph. When $\mathcal{H}=\{H\}$ has only one element, we usually write $ex_\p(n,H)$ instead. Dowden et al. \cite{dowden2016extremal} obtained the sharp bounds $ex_{\mathcal{P}}(n,K_4) =3n-6$ for all $n\geq 4$, $ex_{\mathcal{P}}(n,C_4) \leq\frac{15(n-2)}{7}$ for all $n\geq 4$, and $ex_{\mathcal{P}}(n,C_5) \leq\frac{12n-33}{5}$ for all $n\geq 11$. For $k\in \{4,5\}$, let $\Theta_k$ denote the graph obtained from $C_k$ by adding a chord. Y. Lan et al. \cite{lan2019extremal} showed that $ex_\p(n,\Theta_4)\leq {15(n-2)\over 5}$ for all $n\geq 4$, $ex_\p(n,\Theta_5)\leq {5(n-2)\over 2}$ for all $n\geq 5$. The bounds for $ex_\p(n,\Theta_4)$ and $ex_\p(n,\Theta_5)$ are sharp for infinitely many $n$.  The infinitely often sharp upper bound for $ex_\p(n,C_6)$ was proved by D. Ghosh et al. \cite{ghosh2022planar}. They proved $ex_\p(n,C_6)\leq {5n-14\over 2}$ for all $n\geq 18$. Recently, R. Shi et al. \cite{shi2023planar} and E. Gy\H{o}ri et al. \cite{gyori2023planar} independently proved the sharp bound of $ex_\p(n,C_7)\leq {18\over 7}n-{48\over 7}$ for all $n\geq 60$. E. Gy\H{o}ri et al. \cite{gyori2023planark} also asymptotically determined $ex_{\mathcal{P}}(n,\{K_4,C_5\})$, $ex_{\mathcal{P}}(n,\{K_4,C_6)$ and $ex_{\mathcal{P}}(n,\{K_4,C_7\})$.

A graph is outerplanar if it has a planar drawing for which all vertices belong to the outer face of the drawing. Let $H$ be a graph. The outerplanar Tur\'an number of $H$, denoted by $ex_\p(n,H)$, is the maximum number of edges in an $n$-vertex outerplanar graph which does not contain $H$ as a subgraph. The topics of outerplanar Tur\'an number was initiated by L. Fang et al. in \cite{fang2021outerplanar} where they determined the outerplanar Tur\'an numbers of cycles and paths. 

\begin{theorem}\label{t0}
(\cite{fang2021outerplanar}) Let $n,k$ be integers with $k\geq 3$. If $2\leq n\leq k-1$, then $ex_\p(n,C_k)=2n-3$. If $n\geq k$, let $\lambda=\lfloor {kn-2k-1\over k^2-2k-1}\rfloor+1$. Then, $$ex_\p(n,C_k)=\begin{cases}2n-\lambda-2\lfloor {\lambda\over k}\rfloor-3, &\text{if $k|\lambda$}\\ 2n-\lambda-2\lfloor {\lambda\over k}\rfloor-2, &\text{otherwise.}\end{cases}$$
\end{theorem}

In this paper, we use techniques of dual graph to prove the following theorem:

\begin{theorem}\label{t1}
Let $n,k$ be integers with $n\geq2$, $k\geq 3$. Then, $$ex_\p(n,C_k)\leq \frac{(2k - 5)(kn - k - 1)}{k^2 - 2k - 1}.$$
\end{theorem}

Moreover, for any fixed $k$, we show the bound is sharp for infinitely many $n$:

\begin{theorem}\label{t2}
Let $k\geq 3$ be an integer. Then, for any integer $n$ s.t. $n\equiv k-1\mod k^2-2k-1$, we have $$ex_\p(n,C_k)= \frac{(2k - 5)(kn - k - 1)}{k^2 - 2k - 1}.$$

\end{theorem}
It turns out that Theorem $\ref{t2}$ agrees with Theorem \ref{t0}.

\section{Definitions and Preliminaries}

\begin{definition}[Outerplanar plane graph]
	An \vocab{outerplanar plane graph} is an embedding(drawing) of a graph in which the edge curves do not cross each other and all vertices belong to the outerface.
\end{definition}

\begin{proposition} \label{proposition:edge-maximal:characterization}
	Let \(G = (V, E)\) be an outerplanar plane graph. Then, the following are equivalent:
	\begin{itemize}
		\item \(G\) is an \vocab{edge-maximal outerplanar plane graph}, i.e., no edge can be added to \(G\) while preserving outerplanarity;
		\item \(G\) is \(2\)-connected and all of its innerfaces are triangular;
		\item \(|E| = 2|V| - 3\).
	\end{itemize}
\end{proposition}

\begin{proposition} \label{proposition:edge-maximal:paths}
	Let \(G=(V,E)\) be an edge-maximal outerplanar plane graph.
	Let \((uv)\in E\) be an edge of \(G\) on the outerface. 
	Then, for every \(1 \leq i \leq |V| - 1\), there exists a path between \(u\) and \(v\) of length \(i\).
	Moreover, for every \(3 \leq i \leq |V|\), there exists a cycle of length \(i\).
\end{proposition}

\begin{definition}[Weak dual graph]
	Let \(G\) be an outerplanar plane graph.
	The \vocab{weak dual graph of \(G\)} is the graph whose vertices correspond to the innerfaces of \(G\) and two vertices are connected if and only if their corresponding innerfaces are adjacent. More precisely, the weak dual graph is obtained from the dual graph of $G$ by removing the outerface vertex. Because $G$ is outerplanar, its weak dual graph has no double edges or self-loops.
\end{definition}

\begin{proposition} \label{proposition:weak-dual-acyclic}
	Let \(G\) be an outerplanar plane graph.
	The weak dual graph of \(G\) is acyclic.
\end{proposition}

\begin{definition}[Triangular block]
	Let \(G\) be an outerplanar plane graph.
	A \vocab{triangular block of \(G\)} is a outerplanar plane subgraph \(B\) of \(G\) such that
	\begin{itemize}
		\item \(B\) is an edge-maximal outerplanar plane graph, and;
		\item there is no subgraph \(B'\) of \(G\) such that
			\begin{itemize}
				\item \(B'\) is an edge-maximal outerplanar plane graph, and
				\item \(B\) is a subgraph of \(B'\).
			\end{itemize}
	\end{itemize}
\end{definition}

We remark that triangular blocks may have \(2\) vertices, \(1\) edge and no innerface; these triangular blocks are called \vocab{trivial triangular blocks}.

\begin{proposition}
	Let \(G\) be an outerplanar plane graph.
	Each edge of \(G\) is in exactly one triangular block of \(G\).
	Hence, the collection of triangular blocks of \(G\) is a partition of the edges of \(G\).
\end{proposition}

\begin{definition}[Terminal triangular block]
	Let \(G\) be an outerplanar plane graph.
	A \vocab{terminal triangular block of \(G\)} is a triangular block of \(G\) that shares edges with at most one $(4+)$-innerface of \(G\).
\end{definition}

\begin{lemma} \label{lemma:innerface-terminal}
	Let \(G\) be an outerplanar plane graph with at least one $(4+)$-innerface.
	Then, there exists an innerface of \(G\) of size \(\ell \geq 4\) such that at least \(\ell - 1\) of the triangular blocks containing each of its \(\ell\) edges are terminal triangular blocks.
\end{lemma}

\begin{proof}
	Let \(H\) be the bipartite graph between set of vertices corresponding to $(4+)$-innerfaces of \(G\) and the set of vertices corresponding to triangular blocks of \(G\),
	where we connect a $(4+)$-innerface vertex and a triangular block vertex if, and only if, the $(4+)$-innerface and the triangular block share an edge.

	The graph \(H\) can be obtained from the weak dual graph of \(G\) by contracting all edges between triangles and adding vertices corresponding to trivial triangular blocks of $G$ which is either a leaf (when it's on the outerface), or the middle point of an edge in the weak dual (when it's shared by two innerfaces).
	Proposition~\ref{proposition:weak-dual-acyclic} implies that the weak dual graph of \(G\) is acyclic, hence \(H\) is also acyclic.

	Let \(H'\) be obtained from \(H\) by deleting all vertices corresponding to terminal triangular blocks,
	i.e., by deleting all triangular block vertices of degree at most \(1\) in \(H\).
	Note that the degree of any triangular block vertex in \(H'\) equals its degree in \(H\).
	Therefore, any triangular block vertex in \(H'\) has degree at least \(2\) and, consequently, it is not a leaf in \(H'\).

	Since \(H\) is acyclic, \(H'\) is acyclic.
	Therefore, \(H'\) has a leaf.
	Since no triangular block vertex of \(H'\) is a leaf, there exists a $(4+)$-innerface vertex that is a leaf in \(H'\).
	Therefore, at least \(\ell - 1\) of the \(\ell\) neighbors of this innerface vertex in \(H\) correspond to terminal triangular blocks in $G$, as desired.
\end{proof}

\section{Proof of Theorem \ref{t1}}

\begin{proof}
	We use strong induction on \(n\).
	The base case is $n=2$. Then, $$e\leq 1\leq \frac{(2k - 5)(kn - k - 1)}{k^2 - 2k - 1}={(2k-5)(k-1)\over k^2-2k-1}=1+{(k-2)(k-3)\over k^2-2k-1}.$$ Since $k\geq 3$, we know ${(k-2)(k-3)\over k^2-2k-1}\geq 0$ so the base case holds.

	Let \(G = (V, E)\) be a \(C_k\)-free outerplanar plane graph with \(n\) $(n\geq 2)$ vertices and \(e\) edges.
	Assume, by induction hypothesis, that any \(C_k\)-free outerplanar graph \(G'\) with \(n'\) $(2\leq n'<n)$ vertices and \(e'\) edges satisfies 
	\begin{equation}
		e'(k^2 - 2k - 1) \leq (2k - 5) (kn'-k - 1).
	\end{equation} 

	Note that \(G\) has no innerface of size exactly \(k\).

	\begin{case}
		If \(G\) has an innerface of size at least \(k + 1\), then
		\begin{equation}
			e(k^2 - 2k - 1) \leq (2k - 5) (kn - k - 1).
		\end{equation} 
	\end{case}

	\begin{proof}
		Assume \(G\) has an innerface of size \(\ell\), with \(\ell \geq k + 1\).
		Let \(v_1, \dots, v_\ell\) be the vertices on this face.
		Let \(G_1 = (V_1, E_1), \dots, G_\ell = (V_\ell, E_\ell)\) be plane subgraphs such that
		\(V_i \cap V_{i+1} = \{v_i\}\) for all indices \(1\leq i<\ell\),
		\(V_\ell \cap V_1=\{v_\ell\}\),
		\(V_i \cap V_j = \varnothing\) for all distinct non-consecutive indices \(i, j\),
		\(V_1 \cup \dots \cup V_\ell = V\),
		\(E_i \cap E_j = \varnothing\) for all distinct indices \(i, j\), and
		\(E_1 \cup \dots \cup E_\ell = E\).
		Refer to Figure~\ref{figure:large-face}.

		\begin{figure}[htbp]
			\centering
			\def\svgwidth{7cm}
\begingroup%
  \makeatletter%
  \providecommand\color[2][]{%
    \errmessage{(Inkscape) Color is used for the text in Inkscape, but the package 'color.sty' is not loaded}%
    \renewcommand\color[2][]{}%
  }%
  \providecommand\transparent[1]{%
    \errmessage{(Inkscape) Transparency is used (non-zero) for the text in Inkscape, but the package 'transparent.sty' is not loaded}%
    \renewcommand\transparent[1]{}%
  }%
  \providecommand\rotatebox[2]{#2}%
  \newcommand*\fsize{\dimexpr\f@size pt\relax}%
  \newcommand*\lineheight[1]{\fontsize{\fsize}{#1\fsize}\selectfont}%
  \ifx\svgwidth\undefined%
    \setlength{\unitlength}{340.15748031bp}%
    \ifx\svgscale\undefined%
      \relax%
    \else%
      \setlength{\unitlength}{\unitlength * \real{\svgscale}}%
    \fi%
  \else%
    \setlength{\unitlength}{\svgwidth}%
  \fi%
  \global\let\svgwidth\undefined%
  \global\let\svgscale\undefined%
  \makeatother%
  \begin{picture}(1,0.54166667)%
    \lineheight{1}%
    \setlength\tabcolsep{0pt}%
    \put(0,0){\includegraphics[width=\unitlength,page=1]{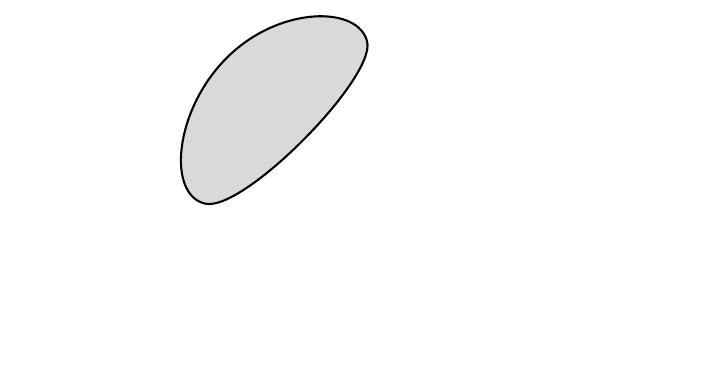}}%
    \put(0.36839759,0.39529155){\color[rgb]{0,0,0}\makebox(0,0)[t]{\lineheight{1.25}\smash{\begin{tabular}[t]{c}$G_1$\end{tabular}}}}%
    \put(0,0){\includegraphics[width=\unitlength,page=2]{drawing3.pdf}}%
    \put(0.63048267,0.39618494){\color[rgb]{0,0,0}\makebox(0,0)[t]{\lineheight{1.25}\smash{\begin{tabular}[t]{c}$G_2$\end{tabular}}}}%
    \put(0,0){\includegraphics[width=\unitlength,page=3]{drawing3.pdf}}%
    \put(0.36929099,0.12761666){\color[rgb]{0,0,0}\makebox(0,0)[t]{\lineheight{1.25}\smash{\begin{tabular}[t]{c}$G_\ell$\end{tabular}}}}%
    \put(0,0){\includegraphics[width=\unitlength,page=4]{drawing3.pdf}}%
  \end{picture}%
\endgroup%

			\caption{Outerplanar plane graph \(G\) and its innerface of size \(\ell \geq k+1\).}
			\label{figure:large-face}
		\end{figure}

		Let \(n_i = |V_i|\), \(e_i = |E_i|\).
		Then, we have \(n + \ell = n_1 + \cdots + n_\ell\) and \(e = e_1 + \cdots + e_\ell\).
		The plane graphs \(G_i\)'s are outerplanar and \(C_k\)-free.
		Hence, 
		\begin{equation}
			e_i(k^2 - 2k - 1) \leq (2k - 5) (kn_i - k - 1).
		\end{equation}

		Add these inequalities for all \(i\), then
		\begin{equation}
			\begin{aligned}
				e(k^2 - 2k - 1) & \leq (2k - 5) (k(n+\ell) - k\ell - \ell) \\
								& \leq (2k - 5) (kn - k - 1),
			\end{aligned}
		\end{equation} 
		as desired.
	\end{proof}
	\begin{case} \label{case:innerface4}
		Assume all innerfaces of \(G\) have size at most \(k - 1\).
		If \(G\) has a $(4+)$-innerface, then
		\begin{equation}
			e(k^2 - 2k - 1) \leq (2k - 5) (kn - k - 1).
		\end{equation} 
	\end{case}
	\begin{proof}
		Lemma~\ref{lemma:innerface-terminal} implies that there exists an innerface of size \(\ell\), with \(4 \leq \ell \leq k-1\), such that at least \(\ell - 1\) of its \(\ell\) surrounding triangular blocks are terminal triangular blocks.

		Denote by \(v_1, v_2, \dots, v_{\ell}\) the vertices of this innerface and denote by \(B_1 = (V_1, E_1), \dots, B_{\ell - 1} = (V_{\ell-1}, E_{\ell-1})\) the terminal triangular blocks surrounding this innerface such that \(v_i, v_{i+1} \in V_i\) for each index $1\leq i< \ell$.
		Let \(G' = (V', E')\) be the plane subgraph consisting of the remaining edges of $G$.
		Refer to Figure~\ref{figure:innerface-terminal}.

		\begin{figure}[htbp]
			\centering
			\makebox[\textwidth]{
				\begin{subfigure}[t]{9cm}
					\centering
					\def\svgwidth{7cm}
					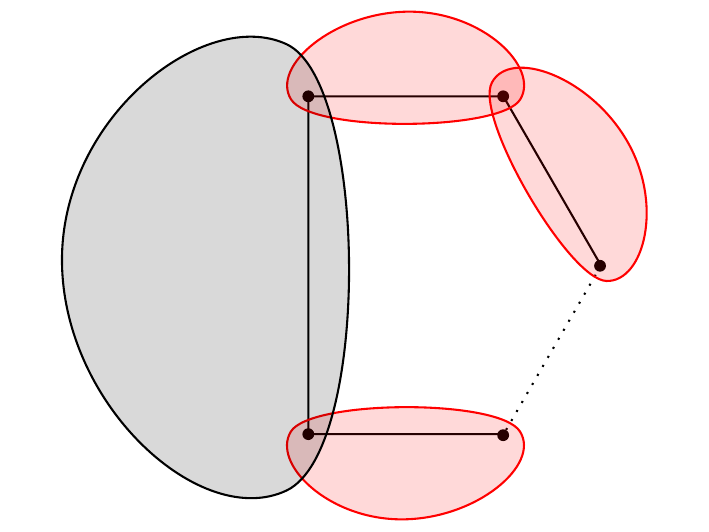
					\caption{Outerplanar plane graph \(G\) and its innerface for which all but one surrounding triangular blocks are terminal triangular blocks.}
					\label{figure:innerface-terminal}
				\end{subfigure}
				\begin{subfigure}[t]{7cm}
					\centering
					\def\svgwidth{7cm}
					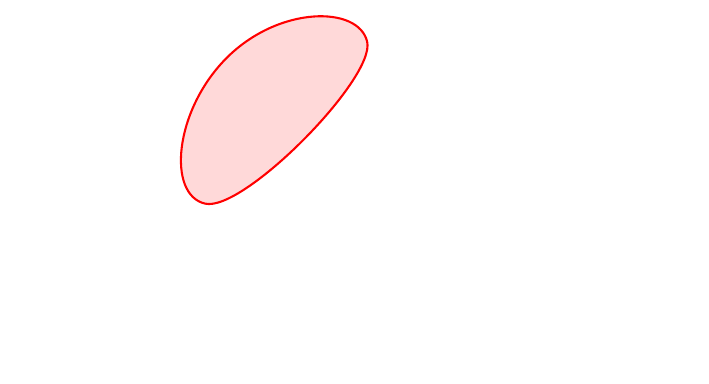
					\caption{Outerplanar plane graph \(H^*\).}
					\label{figure:tighter-face}
				\end{subfigure}
			}
			\caption{Graphs considered on the proof of Case~\ref{case:innerface4}.}
		\end{figure}

		Let \(H\) be the plane subgraph of \(G\) induced by \(V_1 \cup \cdots \cup V_{\ell - 1}\).
		The graph \(H\) is outerplanar.
		Proposition~\ref{proposition:edge-maximal:paths} implies that, for any \(\ell \leq i \leq |H|\), the graph \(H\) contains a cycle of length \(i\).
		Therefore, \(|H| < k\) .

		Let plane graph \(H^* = (V^*, E^*)\) be obtained from \(H\) by contracting the edge \((v_1v_\ell)\).
		The plane graph \(H^*\) is outerplanar, and \(|H^*| =|H|-1< k - 1\).
		Therefore, \(H^*\) has no cycle of length \(k\).
		Refer to Figure~\ref{figure:tighter-face}.

		Let \(n' = |V'|\), \(n^* = |V^*|\), \(e' = |E'|\), \(e^* = |E^*|\).
		Then, we have \(n + 1 = n' + n^*\) and \(e = e' + e^*\).
		The plane graphs \(G'\) and \(H^*\) are outerplanar and \(C_k\)-free. Hence, 
		\begin{gather}
			e'(k^2 - 2k - 1) \leq (2k - 5) (kn' - k - 1), \\
			e^*(k^2 - 2k - 1) \leq (2k - 5) (kn^* - k - 1).
		\end{gather} 
		Add these inequalities, then
		\begin{equation}
			\begin{aligned}
				e(k^2 - 2k - 1) & \leq (2k - 5) (k(n+1) -2k - 2) \\
								& < (2k - 5) (kn -k - 1)
			\end{aligned}
		\end{equation} 
		as desired.
	\end{proof}

	\begin{case}
		If \(G\) is not \(2\)-connected, then
		\begin{equation}
			e(k^2 - 2k - 1) \leq (2k - 5) (kn - k - 1).
		\end{equation} 
	\end{case}
	\begin{proof}
		Assume \(G\) is not \(2\)-connected, i.e., there exist plane subgraphs \(G_1 = (V_1, E_1)\) and \(G_2 = (V_2, E_2)\) satisfying \(|V_1| > 1, |V_2| > 1\), \(|V_1 \cap V_2| \leq 1\), \(V_1 \cup V_2 = V\), \(E_1 \cup E_2 = E\).

		Let \(n_1 = |V_1|\), \(n_2 = |V_2|\), \(e_1 = |E_1|\), \(e_2 = |E_2|\).
		Then, we have \(n + 1 \geq n_1 + n_2\) and \(e = e_1 + e_2\).
		The plane graphs \(G_1, G_2\) are outerplanar and \(C_k\)-free.
		Hence, 
		\begin{gather}
			e_1(k^2 - 2k - 1) \leq (2k - 5) (kn_1 - k - 1), \\
			e_2(k^2 - 2k - 1) \leq (2k - 5) (kn_2 - k - 1).
		\end{gather}
		Add these inequalities, then
		\begin{equation}
			\begin{aligned}
				e(k^2 - 2k - 1) & \leq (2k - 5) (k(n+1) - 2k - 2) \\
								& < (2k - 5) (kn -k - 1),
			\end{aligned}
		\end{equation}
		as desired.
	\end{proof}
	\begin{case}
		Assume \(G\) is \(2\)-connected and all of its innerfaces have size \(3\).
		Then,
		\begin{equation}
			e(k^2 - 2k - 1) \leq (2k - 5) (kn - k - 1).
		\end{equation} 
	\end{case}
	\begin{proof}
		Proposition~\ref{proposition:edge-maximal:characterization} implies that \(G\) is edge-maximal outerplanar plane graph, and \(e = 2n-3\).
		Proposition~\ref{proposition:edge-maximal:paths} implies that, for every \(3 \leq i \leq n\), there is a cycle of length \(i\) in \(G\).
		Hence, \(n \leq k-1\).
		Therefore, 
		\begin{equation}
			\begin{aligned}
				e(k^2 - 2k - 1) & =    (2n - 3)(k^2 - 2k - 1) \\
								& \leq (2k - 5)(kn - k - 1),
			\end{aligned}
		\end{equation}
		as desired.
	\end{proof}

	All possible outerplanar plane graphs \(G\) are covered in the four cases above.
	Therefore, by strong induction, the result follows.
\end{proof}

\section{Proof of Theorem \ref{t2}: Extremal Graph Construction}
\begin{figure}
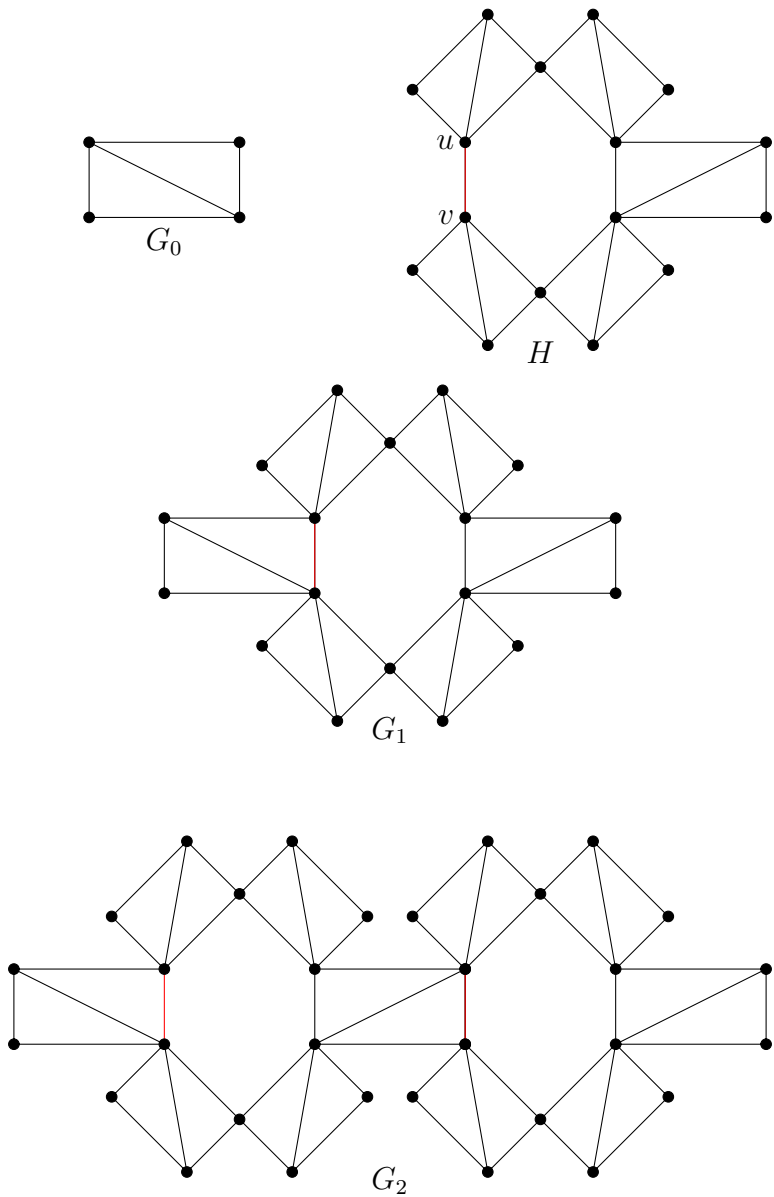

\centering

\bt
\begin{scope}[yshift=3cm,xshift=-3cm]
\df (-1,0.5)circle(2pt)--(1,0.5)circle(2pt)--(1,-0.5)circle(2pt)--(-1,-0.5)circle(2pt);
\draw (-1,-0.5)--(-1,0.5)--(1,-0.5);
\node at (0,-0.5) [anchor=90] {$G_0$};
\end{scope}
\begin{scope}[yshift=3cm,xshift=2cm]
\df(0,1.5)circle(2pt)--(0.7,2.2)circle(2pt)--(1.7,1.2)circle(2pt)--(1,0.5)circle(2pt);
\draw(0.7,2.2)--(1,0.5);
\df(1,0.5)--(3,0.5)circle(2pt)--(3,-0.5)circle(2pt)--(1,-0.5)circle(2pt);
\draw(3,0.5)--(1,-0.5);
\draw (-1,-0.5)node [anchor=0] {$v$}--(-1,0.5)node[anchor=0]{$u$}--(0,1.5)--(1,0.5)--(1,-0.5)--(0,-1.5)--cycle;
\draw[red](-1,-0.5)--(-1,0.5);
\df(0,-1.5)circle(2pt)--(0.7,-2.2)circle(2pt)--(1.7,-1.2)circle(2pt)--(1,-0.5);
\draw(0.7,-2.2)--(1,-0.5);
\df(0,1.5)--(-0.7,2.2)circle(2pt)--(-1.7,1.2)circle(2pt)--(-1,0.5)circle(2pt);
\draw(-0.7,2.2)--(-1,0.5);
\df(0,-1.5)--(-0.7,-2.2)circle(2pt)--(-1.7,-1.2)circle(2pt)--(-1,-0.5)circle(2pt);
\draw(-0.7,-2.2)--(-1,-0.5);
\node at (0,-2)[anchor=90]{$H$};
\end{scope}
\begin{scope}[yshift=-2cm]
\df(0,1.5)circle(2pt)--(0.7,2.2)circle(2pt)--(1.7,1.2)circle(2pt)--(1,0.5)circle(2pt);
\draw(0.7,2.2)--(1,0.5);
\df(1,0.5)--(3,0.5)circle(2pt)--(3,-0.5)circle(2pt)--(1,-0.5)circle(2pt);
\draw(3,0.5)--(1,-0.5);
\draw (-1,-0.5)--(-1,0.5)--(0,1.5)--(1,0.5)--(1,-0.5)--(0,-1.5)--cycle;
\draw[red](-1,-0.5)--(-1,0.5);
\df(0,-1.5)circle(2pt)--(0.7,-2.2)circle(2pt)--(1.7,-1.2)circle(2pt)--(1,-0.5);
\draw(0.7,-2.2)--(1,-0.5);
\df(0,1.5)--(-0.7,2.2)circle(2pt)--(-1.7,1.2)circle(2pt)--(-1,0.5)circle(2pt);
\draw(-0.7,2.2)--(-1,0.5);
\df(0,-1.5)--(-0.7,-2.2)circle(2pt)--(-1.7,-1.2)circle(2pt)--(-1,-0.5)circle(2pt);
\draw(-0.7,-2.2)--(-1,-0.5);
\df(-1,0.5)--(-3,0.5)circle(2pt)--(-3,-0.5)circle(2pt)--(-1,-0.5);
\draw(-3,0.5)--(-1,-0.5);
\node at (0,-2)[anchor=90]{$G_1$};
\end{scope}
\begin{scope}[yshift=-8cm,xshift=-2cm]
\df(0,1.5)circle(2pt)--(0.7,2.2)circle(2pt)--(1.7,1.2)circle(2pt)--(1,0.5)circle(2pt);
\draw(0.7,2.2)--(1,0.5);
\df(1,0.5)--(3,0.5)circle(2pt)--(3,-0.5)circle(2pt)--(1,-0.5)circle(2pt);
\draw(3,0.5)--(1,-0.5);
\draw (-1,0.5)--(0,1.5)--(1,0.5)--(1,-0.5)--(0,-1.5)--(-1,-0.5);
\draw[red](-1,-0.5)--(-1,0.5);
\df(0,-1.5)circle(2pt)--(0.7,-2.2)circle(2pt)--(1.7,-1.2)circle(2pt)--(1,-0.5);
\draw(0.7,-2.2)--(1,-0.5);
\df(0,1.5)--(-0.7,2.2)circle(2pt)--(-1.7,1.2)circle(2pt)--(-1,0.5)circle(2pt);
\draw(-0.7,2.2)--(-1,0.5);
\df(0,-1.5)--(-0.7,-2.2)circle(2pt)--(-1.7,-1.2)circle(2pt)--(-1,-0.5)circle(2pt);
\draw(-0.7,-2.2)--(-1,-0.5);
\df(-1,0.5)--(-3,0.5)circle(2pt)--(-3,-0.5)circle(2pt)--(-1,-0.5);
\draw(-3,0.5)--(-1,-0.5);

\df(4,1.5)circle(2pt)--(4.7,2.2)circle(2pt)--(5.7,1.2)circle(2pt)--(5,0.5)circle(2pt);
\draw(4.7,2.2)--(5,0.5);
\df(5,0.5)--(7,0.5)circle(2pt)--(7,-0.5)circle(2pt)--(5,-0.5)circle(2pt);
\draw(7,0.5)--(5,-0.5);
\draw (3,-0.5)--(3,0.5)--(4,1.5)--(5,0.5)--(5,-0.5)--(4,-1.5)--cycle;
\draw[red](3,-0.5)--(3,0.5);
\df(4,-1.5)circle(2pt)--(4.7,-2.2)circle(2pt)--(5.7,-1.2)circle(2pt)--(5,-0.5);
\draw(4.7,-2.2)--(5,-0.5);
\df(4,1.5)--(3.3,2.2)circle(2pt)--(2.3,1.2)circle(2pt)--(3,0.5)circle(2pt);
\draw(3.3,2.2)--(3,0.5);
\df(4,-1.5)--(3.3,-2.2)circle(2pt)--(2.3,-1.2)circle(2pt)--(3,-0.5)circle(2pt);
\draw(3.3,-2.2)--(3,-0.5);
\node at (2,-2)[anchor=90]{$G_2$};
\end{scope}
\et
\caption{\label{f4}Example of construction when $k=5$}
\end{figure}
\begin{proof}
We will do induction on the construction. Let $k\geq 3$ be a fixed integer. The base case $G_0$ can be any edge-maximal outerplanar plane graph with $k-1$ vertices and $2k-5$ edges. Now we construct outerplanar graph $H$ in this way: first we draw a face of size $k+1$, then randomly pick $k$ edges from it, and for each of these edges we attach it to a copy of $G_0$. Let the other edge on that $(k+1)$-face be called $(uv)$. In short, $H$ has $k^2-2k+1$ vertices and $2k^2-5k+1$ edges, with exactly $k$ triangular blocks of $k-1$ vertices, $1$ trivial triangular block and no other triangular blocks. Given $G_m$, the inductive step from $G_m$ to $G_{m+1}$ is the following: draw a copy of $H$ and merge $(uv)$ with any edge on the boundary of $G_m$. Let $G_m=(V_m,E_m)$, then 
\begin{align*}
|V_m|&=k-1+(m+1)(k^2-2k-1)\\
|E_m|&=(2k-5)(k+1+mk)\\
|E_m|&={(2k-5)(k|V_m|-k-1)\over k^2-2k-1}.
\end{align*}

In Figure \ref{f4}, we show an example of our construction when $k=5$.

\end{proof}

\sloppy
\bibliographystyle{abbrv}
\bibliography{outerplanarck}
\end{document}